\documentclass[11pt, a4paper]{article}%
\usepackage{amsmath,amssymb,eucal}
\usepackage{amsfonts}
\usepackage{mathrsfs}
\usepackage{slashed}
\usepackage{graphicx}
\usepackage{ifpdf}
 \usepackage{url}
\usepackage[all,knot,poly]{xy}
\numberwithin{equation}{section} 
\def\beq{\begin{equation}}
\def\eeq{\end{equation}}

\def\bR{ {{\mathbb{R}}}}

\newcommand{\pd}{{\rm DP}}
\newcommand{\sgn}{{\rm sgn}}

\newcommand{\pk}[1]{p_{\kappa}}

\newcommand{\ol}{\overline{l}}
\newcommand{\ul}{\underline{l}}
\newcommand{\oL}{\overline{L}}
\newcommand{\uL}{\underline{L}}
\newcommand{\on}{\overline{n}}
\newcommand{\un}{\underline{n}}

\newtheorem{defn}{{\bf Definition}}[section]
\newtheorem{thm}[defn]{{\bf Theorem}}
\newtheorem{cor}[defn]{{\bf Corollary}}
\newtheorem{lem}[defn]{{\bf Lemma}}
\newtheorem{prop}[defn]{{\bf Proposition}}
\newtheorem{rem}[defn]{{\bf Remark}}

\newtheorem{notation}[defn]{Notation}
\newenvironment{proof}[1][Proof]{\textbf{#1.} }{\hfill \rule{0.5em}{0.5em}}
\begin{document}

\title{Invariants in Quantum Geometry}
\author{Adrian P. C. Lim \\
Email: ppcube@gmail.com
}

\date{}

\maketitle

\begin{abstract}
In quantum geometry, we consider a set of loops, a compact orientable surface and a solid compact spatial region, all inside $\mathbb{R} \times \mathbb{R}^3 \equiv \mathbb{R}^4$, which forms a triple. We want to define an ambient isotopic equivalence relation on such triples, so that we can obtain equivalence invariants. These invariants describe how these submanifolds are causally related to or `linked' with each other, and they are closely associated with the linking number between links in $\mathbb{R}^3$. Because we distinguish the time-axis from spatial subspace in $\mathbb{R}^4$, we see that these equivalence relations, will also imply causality.
\end{abstract}

\hspace{.35cm}{\small {\bf MSC} 2010: } 51H20, 57Q45  \\
\indent \hspace{.35cm}{\small {\bf Keywords}: Quantum geometry, Causality, Linking number, Loops, Surface}




Quantum geometry describes how quantum matter interacts with geometry in \cite{Ashtekar:2000eq}. One of the key players in this theory is loops, which will yield topological invariants of loops, as described in \cite{PhysRevLett.61.1155}. But unlike a topological theory, Smolin in \cite{Smolin:2005mq} remarked that one should not focus on the topology on the ambient space, but rather how the events are causally linked.

We consider our 4-manifold to be $\bR^4 \equiv \mathbb{R} \times \mathbb{R}^3 $, which will be our ambient space. The time-axis is distinguished from spatial subspace. A loop is a continuous simple closed curve in $\bR \times \bR^3$, considered as a 1-dimensional manifold. We also consider an orientable compact surface, with or without boundary, and a compact solid region, viewed as a 3-dimensional submanifold, all inside our ambient space.

We want to define invariants that describe how these submanifolds are `linked' together in the ambient space. Linking is a topological concept; two loops are linked together if it is impossible to translate one loop by an arbitrary distance from the other loop without the two objects actually crossing one another. But in $\bR^4$, one can topologically deform and `unlink' the loops, without crossing. As such, if we use ambient isotopy as an equivalence relation, then an equivalence class of a set of loops in $\bR^4$ will be a set of `unlinked' simple closed curves.

\section{Hyperlinks in $\bR^4$}

Consider our ambient space $\bR^4 \equiv \bR \times \bR^3$, whereby $\bR$ will be referred to as the time-axis and $\bR^3$ is the spatial 3-dimensional Euclidean space. Fix the standard coordinates on $\bR^4\equiv \bR \times \bR^3 $, with time coordinate $x_0$ and spatial coordinates $(x_1, x_2, x_3)$. Let $\pi_0: \bR^4 \rightarrow \bR^3$ denote this projection.

Let $\{e_i\}_{i=1}^3$ be the standard basis in $\bR^3$. And $\Sigma_i$ is the plane in $\bR^3$, containing the origin, whose normal is given by $e_i$. So, $\Sigma_1$ is the $x_2-x_3$ plane, $\Sigma_2$ is the $x_3-x_1$ plane and finally $\Sigma_3$ is the $x_1-x_2$ plane.

Note that $\bR \times \Sigma_i \cong \bR^3$ is a 3-dimensional subspace in $\bR^4$. Here, we replace one of the axis in the spatial 3-dimensional Euclidean space with the time-axis. Let $\pi_i: \bR^4 \rightarrow \bR \times \Sigma_i$ denote this projection.

For a finite set of non-intersecting simple closed curves in $\bR^3$ or in $\bR \times \Sigma_i$, we will refer to it as a link. If it has only one component, then this link will be referred to as a knot. A simple closed curve in $\bR^4$ will be referred to as a loop. A finite set of non-intersecting loops in $\bR^4$ will be referred to as a hyperlink in this article. We say a link or hyperlink is oriented if we assign an orientation to its components.

We need to consider the space of hyperlinks in $\bR \times\bR^3$, which is too big for our consideration. Given any hyperlink in $\bR \times\bR^3$, it is ambient isotopic to the trivial hyperlink, a union of `unlinked' loops. So the equivalence class of ambient isotopic hyperlinks will give us only the `unlinked' hyperlink, which is too trivial. Hence we will instead consider a special equivalence class of hyperlinks.

\begin{defn}(Time-like separation)\\
Let $\vec{x}, \vec{y}$ be 2 points in $\bR \times \bR^3$, with coordinates $(x_0, x_1, x_2, x_3)$ and $(y_0, y_1, y_2, y_3)$ respectively. We say $\vec{x}$ and $\vec{y}$ are time-like separated if the Minkowski distance between them is \beq \sum_{i=1}^3 (x_i - y_i)^2 - (x_0 - y_0)^2 < 0. \nonumber \eeq We also say $\vec{x}$ and $\vec{y}$ are space-like separated if the Minkowski distance between them is \beq \sum_{i=1}^3 (x_i - y_i)^2 - (x_0 - y_0)^2 > 0. \nonumber \eeq
\end{defn}

When 2 points are time-like separated, we see that their time components must be different. Notice that we use the Minkowski metric to define the time-like separation. In Quantum Field Theory, one uses the Minkowski metric. But in General Relativity, this is no longer correct as the metric is actually a variable and the stress-energy tensor would determine the correct metric from solving the Einstein's equations.

A quantized theory of gravity should be independent of any background metric. Therefore, one should not use Minkowski metric and since there is no notion of a preferred metric, there is no such thing as time-like separation. See \cite{Thiemann:2002nj} and \cite{Mercuri:2010xz}. Nevertheless, we will still borrow the term `time-like' and define the following special class of hyperlinks we would like to consider.

\begin{defn}(Time-like hyperlink)\label{d.tl.1}\\
Let $L$ be a hyperlink. We say it is a time-like hyperlink if given any 2 distinct points $\vec{x}\equiv (x_0, x_1, x_2, x_3), \vec{y}\equiv (y_0, y_1, y_2, y_3) \in L$, $\vec{x} \neq \vec{y}$,
\begin{enumerate}
  \item (T1) $\sum_{i=1}^3(x_i - y_i)^2 > 0$;
  \item (T2) if there exists $i, j$, $i \neq j$ such that $x_i = y_i$ and $x_j = y_j$, then $x_0 - y_0 \neq 0$.
\end{enumerate}
\end{defn}

We make the following remarks, which is immediate from the definition.

\begin{rem}\label{r.p.1}
\begin{enumerate}
  \item In Condition T1, we insist that any 2 distinct points in a hyperlink are also separated, when projected using $\pi_0$, in 3-dimensional spatial space $\bR^3$. This is to ensure that when we project the hyperlink in $\bR^3$, we obtain a link.
  \item Conditions T1 and T2 imply that given a hyperlink $L$, for each $i=1,2,3$, $\pi_i(L) \in \bR \times \Sigma_i$ is a link. Furthermore, they guarantee that for each crossing as defined in Subsection \ref{ss.ld}, its algebraic crossing number and its time-lag are well-defined.
\end{enumerate}
\end{rem}

\begin{defn}\label{d.ts.1}
Two oriented time-like hyperlinks $L$ and $L'$ in $\bR \times \bR^3$ are time-like isotopic to each other if there is an orientation preserving continuous map $F: \bR \times \bR^3 \times [0,1] \rightarrow \bR \times \bR^3$, such that
\begin{enumerate}
  \item $F_0$ is the identity map;
  \item $F_t$ is a homeomorphism from $\bR \times \bR^3$ to $\bR \times \bR^3$;
  \item $F_1(L) = L'$;
  \item each $F_t(L)$ is a time-like hyperlink.
\end{enumerate}
In other words, two time-like hyperlinks $L_1$ and $L_2$ in $\bR \times \bR^3$ are time-like isotopic if $L_1$ can be continuously deformed to $L_2$ while remaining time-like.
 \end{defn}

\begin{rem}
By definition, $L$ and $L'$ in $\bR \times \bR^3$ are time-like isotopic to each other implies that
\begin{enumerate}
  \item $\pi_0(L)$ and $\pi_0(L')$ are ambient isotopic to each other in $\bR^3$;
  \item $\pi_i(L)$ and $\pi_i(L')$ are ambient isotopic to each other in $\bR \times \Sigma_i$, $i=1,2,3$.
\end{enumerate}
\end{rem}

Throughout this article, all our hyperlinks will be time-like and we consider equivalence classes of such hyperlinks using Definition \ref{d.ts.1}.

\begin{defn}\label{d.tau.1}
For any $\vec{p} = (p_0, p_1, p_2, p_3) \in \bR \times \bR^3$, we define $\tau(\vec{p}) = p_0$.
Given two loops $\ol$ and $\ul$, we say that $\ol < \ul$ if for any $\vec{p} \in \ol$, any $\vec{q} \in \ul$, we have that $\tau(\vec{p}) < \tau(\vec{q})$. When $\ol < \ul$, we say that the loop $\ol$ occurs before the loop $\ul$. If $\ol > \ul$, we say that the loop $\ol$ occurs after the loop $\ul$.

\end{defn}

\subsection{Link Diagrams}\label{ss.ld}

The next thing that we want to define is the hyperlinking number of a hyperlink. This should be thought of as a generalization of the linking number of a link, as described in Definition 4.3 in \cite{CS-Lim01}. Indeed, one can calculate the hyperlinking number from a link diagram.

Any link in $\mathbb{R}^3$ can be represented by a link diagram in $\bR^2$, up to isotopy. This allows us to study links using link diagrams.
Two link diagrams $D$ and $D'$ are (planar) isotopic if there exists an isotopy $h$ of $\mathbb{R}^2$ such that $h(1,D) = D'$. To check if $D$ and $D'$ are isotopic, it suffices to show that $D$ can be obtained from $D'$ by a sequence of Reidemeister Moves.

A crossing $p$ on an oriented link diagram is represented (up to isotopy) either by
\beq
\xygraph{
     !{0;/r1.0pc/:}
     [u(0.5)]
     !{\xoverv=<}
   }
\ \ {\rm or}\ \
\xygraph{
     !{0;/r1.0pc/:}
     [u(0.5)]
     !{\xunderv=<}
   }
. \nonumber \eeq We assign the value $\varepsilon(p):= +1$ for the diagram on the left; $\varepsilon(p):= -1$ for the diagram on the right. Note that $\varepsilon(p)$ is known as the algebraic crossing number of the crossing $p$.

Given 2 oriented simple closed curves $\ol$ and $\ul$ which are non-intersecting in $\bR^3$, project it on $\Sigma_i$ to form a link diagram. Write $\pd(\Sigma_i ;\ \ol, \ul)$ to denote the set of all crossings in a link diagram of curves $\ol$ and $\ul$. Define the linking number between $\ol$ and $\ul$, \beq {\rm lk}(\ol, \ul) := \sum_{p \in \pd(\Sigma_i ;\ \ol, \ul)}\varepsilon(p). \nonumber \eeq The linking number between 2 oriented knots in $\bR^3$ is an invariant up to ambient isotopy, so it does not matter which plane we project it onto.

Recall a hyperlink is a finite set of non-interesting simple closed curves in $\bR \times \bR^3$ and considered as time-like, as defined in Definition \ref{d.tl.1} and oriented. We can project an oriented hyperlink on $\Sigma_i$ to form a link diagram as before. Suppose each crossing $p$ on a link diagram is formed from projecting 2 arcs $\overline{C}$ and $\underline{C}$ in loops $\ol$ and $\ul$ respectively.

Let $\vec{x}=(x_0, x) \in \overline{C}$ and $\vec{y}=(y_0, y) \in \underline{C}$ respectively, with $x, y \in \bR^3$, such that $p$ is the projection of $x$ and $y$ onto the plane $\Sigma_i$. Note that $x_0$ and $y_0$ are the time components of $\vec{x}$ and $\vec{y}$ respectively.

Define the time-lag of $p$ by
\beq \sgn(p;\ol_0 : \ul_0) = \left\{
                                  \begin{array}{ll}
                                    1, & \hbox{$x_0 < y_0$;} \\
                                    -1, & \hbox{$x_0 > y_0$,}
                                  \end{array}
                                \right. \nonumber \eeq
and the hyperlinking number between $\ol$ and $\ul$ as \beq
{\rm sk}(\ol, \ul) := \sum_{k=1}^3\sum_{p \in \pd(\Sigma_k; \ol, \ul)}\varepsilon(p)\cdot \sgn(p;\ol_0 : \ul_0). \nonumber \eeq

\begin{rem}\label{r.r.1}
Let $L$ be an oriented time-like hyperlink. Note that the crossing number $\varepsilon(p)$ for a crossing $p$ in a link diagram in $\Sigma_k$ depends on projecting $\pi_0(L)$ onto $\Sigma_k$, $k=1,2,3$. The time-lag of the same crossing $p$ depends on projecting $\pi_k(L) \in \bR \times \Sigma_k$ onto $\Sigma_k$.
\end{rem}

Unlike the linking number, the hyperlinking number is not an invariant under time-like isotopy. To make it an invariant, we need to impose an extra condition on a hyperlink.

\begin{defn}(Time-ordered pair of hyperlinks)\label{d.tl.2}\\
Suppose we have an oriented time-like hyperlink, denoted as $\chi(\oL, \uL)$, consisting of 2 non-empty sets of oriented time-like hyperlink, denoted as $\oL= \{\ol^1, \cdots, \ol^{\overline{n}}\}$ and $\uL= \{\ul^1, \cdots, \ul^{\underline{n}}\}$ respectively.

Pick a component loop $\ol^u \in \oL$ and another component loop $\ul^v \in \uL$. Refer to Definition \ref{d.tau.1}. We require that either
$\ol^u < \ul^v$, or $\ol^u > \ul^v$. In this case, we say that the oriented time-like hyperlink $\chi(\oL, \uL)$ consists of an oriented time-ordered pair of time-like hyperlinks.
\end{defn}

If we can order $\ol^u$ and $\ul^v$, then the hyperlinking number between $\ol^u$ and $\ul^v$ will be
\beq {\rm sk}(\ol^u, \ul^v) =
\left\{
  \begin{array}{ll}
    3 \times {\rm lk}(\pi_0(\ol^u), \pi_0(\ul^v)), & \hbox{$\ol^u < \ul^v$;} \\
    -3\times {\rm lk}(\pi_0(\ol^u), \pi_0(\ul^v)), & \hbox{$\ol^u > \ul^v$.}
  \end{array}
\right. \nonumber \eeq Note that ${\rm sk}(\ol^u, \ul^v) = -{\rm sk}(\ul^v, \ol^u)$.

\begin{defn}\label{d.tl.3}
Let $\chi(\oL, \uL)$ be an oriented time-like hyperlink which consists of a pair of oriented time-ordered hyperlinks $\oL= \{\ol^1, \cdots, \ol^{\overline{n}}\}$ and $\uL= \{\ul^1, \cdots, \ul^{\underline{n}}\}$ as defined in Definition \ref{d.tl.2}. We say that $\chi(\oL, \uL)$ is time-like isotopic to an oriented hyperlink $\chi(\oL', \uL')$, preserving the time-ordering, if there exists a continuous map $F$ as defined in Definition \ref{d.ts.1}, such that we have $F_t(\ol^u) < F_t(\ul^v)$ or $F_t(\ol^u) > F_t(\ul^v)$ for all $t$.
\end{defn}

\begin{rem}\label{r.t.1}
Suppose $\ol^u < \ul^v$. Then we must have $F_t(\ol^u) < F_t(\ul^v)$ for all $t$.
\end{rem}

We will now consider equivalence classes of time-like hyperlinks $\chi(\oL, \uL)$, which consists of a time-ordered pair. First, note that this will not include all time-like hyperlinks, as the component loops may not be time-ordered. Those pair of hyperlinks which cannot be ordered as described in Definition \ref{d.tl.2} will not be in the equivalence class.

Second, suppose we have a pair of loops, $(\ol, \ul)$, such that under time-like isotopy, we can change the ordering from $\ol < \ul$ to $\ol > \ul$. Consider $L = (\ol, \ul)$ is a time-like hyperlink, with $\ol < \ul$ and $L'$ is time-like isotopic to $L$, but with $\ol > \ul$. Then, we will treat $L$ and $L'$ to be inequivalent under the above relation.

\begin{rem}
This time-ordering between loops will imply causality. The equivalence relation says that we are not allowed to consider homeomorphisms in space-time that violates causality.
\end{rem}

Under time-like isotopy that preserves the time-ordering, the hyperlinking number between equivalent classes of $\ol^u$ and $\ul^v$ will now be an invariant.

\begin{thm}\label{t.w.1}
Consider two oriented time-like hyperlinks, $\oL = \{\ol^u \}_{u=1}^{\on}$, $\uL = \{\ul^v \}_{v=1}^{\un}$ in $\bR \times \bR^3$ with non-intersecting (closed) loops, which together form a new oriented time-ordered pair of time-like hyperlinks, denoted by $\chi(\oL, \uL)$.

Define the hyperlinking number between $\ol^u$ and $\uL$ as \beq {\rm sk}(\ol^u, \uL) := \sum_{v=1}^{\un}{\rm sk}(\ol^u, \ul^v), \nonumber \eeq
and the hyperlinking number between
$\oL$ and $\uL$, \beq {\rm sk}(\oL, \uL) := \sum_{u=1}^{\on}{\rm sk}(\ol^u, \uL), \nonumber\eeq both calculated from $\chi(\oL, \uL)$.

Consider now the equivalence class of $\chi(\oL, \uL)$, under time-like isotopy which preserves the time-ordering as given in Definition \ref{d.tl.3}. Then the hyperlinking number between $\oL$ and $\uL$  is invariant under this equivalence relation.
\end{thm}

\begin{proof}
It suffices to show that for each $u = 1, \cdots, \on$ and each $v = 1, \cdots, \un$, ${\rm sk}(\ol^u, \ul^v)$ is invariant under the equivalence relation. Suppose $\ol^u < \ul^v$. As discussed earlier, ${\rm sk}(\ol^u, \ul^v) = 3 \times {\rm lk}(\pi_0(\ol^u), \pi_0(\ul^v))$.

Now, ${\rm lk}(\pi_0(\ol^u), \pi_0(\ul^v))$ is invariant under any time-like isotopy of the hyperlink containing $\ol^u$ and $\ul^v$. By Remark \ref{r.t.1}, the time-lag for all the crossings in a link diagram projected on $\Sigma_i$, $i=1,2,3$, for any continuous deformation of $\pi_0(\ol^u)$ and $\pi_0(\ul^v)$ will be the same, and it will never change sign. Hence the hyperlinking number ${\rm sk}(\ol^u, \ul^v)$ remains invariant under time-like isotopy, as long as the time-ordering of the continuous deformation of $\ol^u$ and $\ul^v$ is also preserved throughout the time-like isotopy.

\end{proof}

\section{Surfaces}

Choose an orientable, compact smooth surface $S \subset \bR^4$, with or without boundary. When $S$ has no boundary, we will henceforth call it closed. Pick a time-like loop $l \subset \bR^4$, assumed to be $C^1$, and disjoint from $S$. From page 43 in \cite{greensite2011introduction}, it is possible to define a linking number between a closed surface $S$ and a loop.

We want to define a linking number by projecting $S$ and $l$ inside $\bR^3$ using $\pi_0$. However, when we project $S$ inside $\bR^3$, $\pi_0(S)$ may not be a surface. Since we want to consider ambient isotopic equivalence classes of surface, and for the equivalent class of surface which we are interested in the last section of this article, we shall require that the surface in consideration $S \subset \bR^4$ satisfies the condition that $\pi_0: S \rightarrow \pi_0(S)$ is injective. Furthermore, we assume that $\pi_0(l)$ intersects $\pi_0(S)$ at finitely many points and we do not allow $\pi_0(l)$ to be tangent to $\pi_0(S)$ at any such intersection points.

\begin{rem}
The condition that $\pi_0: S \rightarrow \pi_0(S)$ is injective is to ensure that we project $S$ `nicely' inside $\bR^3$, without any two different patches on the surface $S$ `collapsing' onto the same patch on $\pi_0(S)$.
\end{rem}

Let $\pd(\pi_0; l, S)$ denote the set of finitely many intersection points between $\pi_0(l)$ and $\pi_0(S)$ as described in the preceding paragraph, henceforth termed as piercings. Choose a non-zero normal $n_S$ on $\pi_0(S)$. Orientate $\pi_0(l)$ and let $\nu_l$ be the non-zero tangent vector along the oriented curve $\pi_0(l)$ in $\bR^3$. For a piercing $p$, define the orientation of $p$, $\sgn(p; l, S)$, which takes the value +1 if $n_S(p)\cdot \nu_l(p) > 0$; $-1$ if $n_S(p)\cdot \nu_l(p) < 0$. This choice of orientation is consistent with the volume form $dx_1 \wedge dx_2 \wedge dx_3$ on $\bR^3$.

For each such piercing $p$, let $(x_0, x) \in l$, $(y_0, y) \in S$ such that $p = \pi_0(x_0,x) = x$ and $p = \pi_0(y_0, y)$, so $p = x = y \in \bR^3$. Define the height of $p$,
\beq {\rm ht}(p; l, S) =
\left\{
  \begin{array}{ll}
    1, & \hbox{$x_0 < y_0$;} \\
    -1, & \hbox{$x_0 > y_0$.}
  \end{array}
\right.
 \nonumber \eeq

Define the algebraic piercing number of $p$ as \beq \varepsilon(p) :=  \sgn(p; l, S)\cdot {\rm ht}(p; l, S). \nonumber \eeq And we define the linking number between $l$ and $S$ as \beq {\rm lk}(l, S) := \sum_{p \in \pd(\pi_0; l,S)}\varepsilon(p). \nonumber \eeq

\begin{rem}
The linking number defined in \cite{EH-Lim03}, is negative of the definition given here.
\end{rem}

When $S$ is closed, we will show later that this linking number is an invariant under ambient isotopy of $l$ and $S$. See Theorem \ref{t.w.2}. When $S$ has a boundary, then it is no longer true that the linking number is an invariant under ambient isotopy.

When $S$ is closed and $\pi_i: S \rightarrow \pi_i(S)$ is injective, we can compute the linking number by projecting $l$ and $S$ inside $\bR \times \Sigma_i$, for $i=1,2,3$, using $\pi_i$. We will also assume that $\pi_i(l)$ intersect $\pi_i(S)$ at a finite number of intersections and we do not allow $\pi_i(l)$ to be tangent to $\pi_i(S)$ at any such intersection points, henceforth also known as piercings.

The height of any piercing $p \in \pi_i(S) \cap \pi_i(l)$ will now be defined using $x_i$ and $y_i$, positive when $x_i < y_i$; negative otherwise, for $(x_0,x_1, x_2, x_3) \in l$ and $(y_0, y_1, y_2, y_3) \in S$. The orientation will be defined using the orientation of the curve $\pi_i(l)$ and the non-zero normal on $\pi_i(S)$, which should be consistent with the orientations on $\bR \times \Sigma_1$, $\bR \times \Sigma_2$ and $\bR \times \Sigma_3$, determined by volume forms $dx_0 \wedge dx_3 \wedge dx_2$, $dx_0 \wedge dx_1 \wedge dx_3$ and $dx_0 \wedge dx_2 \wedge dx_1$ respectively. That is, if $\nu_l(p)$ and $n_S(p)$ are the tangent vector of $\pi_i(l)$ and normal to $\pi_i(S)$ respectively at $p = \pi_i(S) \cap \pi_i(l)$, then the orientation of $p$ is assigned +1 if $\nu_l(p) \cdot n_S(p) < 0$; -1 if $\nu_l(p) \cdot n_S(p) >0$.

The algebraic piercing number of $p$ will now be defined as the product of its height and its orientation. By summing up the algebraic numbers of all its piercings in $\pi_i(S)$, it will be equal to the linking number defined using $\pi_0$.

Let $c \geq 0$ and consider the plane $P_c\subset \bR^4$ defined by the linear equation $x_0 = c(1-x_1 - x_2 - x_3)$. The map $T_c: \bR^3 \rightarrow P_c \subset \bR^4$, given by \beq T_c: (x,y,z) \in \bR^3 \longmapsto \left(c(1-x-y-z), x,y,z  \right) \in P_c \label{e.tc.1} \eeq is a homeomorphism for $c \geq 0$. When $c = 0$, we view $T_c$ as an identity map on $\bR^3$, identified with $\{0\} \times \bR^3$.

In the event that $S \subset \{0\} \times \bR^3$, we will map $S$ inside $P_c$ using $T_c$, for $c>0$. Then $\pi_0: T_c(S) \rightarrow \bR^3$ and $\pi_i: T_c(S) \rightarrow \bR \times \Sigma_i$, $i=1,2,3$, are all injective. We will define the linking number between a time-like loop $l$ disjoint from $S \subset \bR^3$, as the limit of the linking number between $l$ and $T_c(S)$, as $c$ goes down to 0.

\subsection{Surfaces without boundary}

We consider $S$ is a closed surface in $\bR^4$ and $l$ is a time-like loop, disjoint from $S$. To show that the linking number is invariant under ambient isotopy of both $l$ and $S$, it suffices to consider $S \subset  \{0\} \times \bR^3$. Note that $\sup_{\vec{x} \in S}\tau(\vec{x}) = \inf_{\vec{x} \in S}\tau(\vec{x}) = 0$.

Recall we distinguished the time-axis in $\bR^4$ and we defined $\tau$ in Definition \ref{d.tau.1}. An arc will be referred to as an open, connected subset inside a curve throughout this article.

\begin{defn}(Interior and exterior)\\
Consider a connected, closed surface $S_0$ in $\bR^3$. Then $S_0$ will divide $\bR^3$ into 2 open connected sets, one is bounded and the other is unbounded. We will refer the bounded set as the interior of $S_0$; the latter as the exterior of $S_0$.

When $S_0 = \bigcup_{k=1}^n \tilde{S}_k \subset \bR^3$ is a disconnected closed surface, the interior of $S_0$ will be the union of the interior of all its connected components $\tilde{S}_k$.
\end{defn}

\begin{defn}(Movement $W$)\\
Suppose we have an arc $C_0$, and it lies in the interior of a connected closed surface $S_0$, except its two end points $p$ and $q$, which intersect $S_0$, all inside $\bR^3$. We can slowly pull the two ends of the arc, withdrawing it into the exterior of $S_0$, such that the arc continues to intersect $S_0$ at two intersection points. Continue this process until the arc is tangent to the surface, or just touches the surface at only one intersection point. If we further withdraw the arc, it will now lie totally in the exterior, disjoint from the surface $S_0$.

Suppose we have an arc $C$ and a connected closed surface $S$ in $\bR^4$. Further assume that $S = \pi_0(S)$. We will term a continuous deformation of $C$ as Movement $W$ between $C$ and $S$, if $C$ is ambient isotopic to $C'$ in $\bR^4$, with
\begin{itemize}
  \item during the ambient isotopy between $C$ and $C'$, the intermediate deformed arcs between $C$ and $C'$ remains disjoint with $S$;
  \item this continuous deformation of arcs between $C$ and $C'$ projects down into $\bR^3$ using $\pi_0$, and it gives us a continuous deformation between $\pi_0(C)$ and $\pi_0(C')$ as described in the preceding paragraph, with $\pi_0(C)$ and $\pi_0(C')$ lying in the interior and exterior of $\pi_0(S)$ respectively.
\end{itemize}
\end{defn}

\begin{rem}
\begin{enumerate}
  \item Note that when we apply Movement $W$, we are removing two piercings between $\pi_0(C)$ and $\pi_0(S)$.
  \item In the definition of Movement $W$, we consider $C$ to be inside the interior of $\pi_0(S)$. But we can also consider $C$ to be in the exterior of $\pi_0(S)$ and withdraw it to be in the interior of $\pi_0(S)$.
\end{enumerate}
\end{rem}

\begin{lem}\label{l.w.1}
Let $C$ be an arc inside a time-like loop $l$, i.e. $C$ is an open, connected curve in the loop. Suppose $S$ is a closed surface in $\{0\} \times \bR^3$ and $C$ intersects at two piercings in $S$. We can apply Movement $W$, to remove these piercings, if and only if we have either $\tau(\vec{x}) < 0$, or $\tau(\vec{x}) > 0$, for every $\vec{x} \in C$.
\end{lem}

\begin{proof}
Let $S' \subset S$ be homeomorphic to an open disc, and it contains two piercings. Without any loss of generality, suppose $\pi_0(C)$ lies in the interior, disjoint from $\pi_0(S)$ and for every $\vec{x} \in C$, we have $\tau(\vec{x}) > 0$. Thus, for any $\vec{x} \in C$, we have $\tau(\vec{x}) > \tau(\vec{y})=0$ for any $\vec{y} \in S'$.

The closed arc $\overline{\pi_0(C)}$ intersects the surface $\pi_0(S')$ at two intersection points, call them $p \in \pi_0(S')$ and $q \in \pi_0(S')$ respectively. Note that $p$ and $q$ correspond to points $\vec{p} = (p_0, p) \in C$ and $\vec{q} = (q_0,q)\in C$ respectively, which are the end points of $C$.

Because $\tau(\vec{p}), \tau(\vec{q}) > 0$, we can continuously deform $C \subset \bR\times \bR^3$ into $C' \subset \bR\times \bR^3$, always keeping the time component bigger than $\tau(\vec{y})$ for any $\vec{y} \in S'$ during the deformation process, and any intermediate deformed curve between $C$ and $C'$ remains disjoint from $S'$. Because any projected deformed curve only intersects $\pi_0(S') \equiv S' \subset S$, we see that the deformed curve must remain disjoint from $S$. This allows us to apply Movement $W$ and remove the piercings. The argument for the case when $\tau(\vec{x}) < 0$ is similar.

Now to prove the other direction. Suppose we can apply Movement $W$. During the Movement $W$ process, the arc $C$ will be continuously deformed to an arc $C'$ such that $\pi_0(C')$ is just tangent to the surface $\pi_0(S)$. Since during the process, we must have $C'$ is disjoint from $S$, we see that there must exist a neighborhood $S' \subset S$, homeomorphic to an open disc in $\bR^2$, such that either $\tau(\vec{x}) > \sup_{\vec{p} \in S'}\tau(\vec{p})$, or $\tau(\vec{x}) < \inf_{\vec{p} \in S'}\tau(\vec{p})$, for $\pi_0(\vec{x})$ being the intersection point when $\pi_0(C')$ just touches $\pi_0(S')$. Since Movement $W$ is a continuous deformation, hence there must exist a $C''$, ambient isotopic to $C'$, such that $\pi_0(C'')$ intersects $\pi_0(S')$ at two piercings and either $\tau(\vec{x}) < \inf_{\vec{p} \in S'}\tau(\vec{p})=0$, or $\tau(\vec{x}) > \sup_{\vec{p} \in S'}\tau(\vec{p})=0$, for every $\vec{x} \in C''$. This completes the proof.
\end{proof}

\begin{cor}\label{c.w.1}
Movement $W$ is not permissible for an arc $C$, if and only if the piercings $\pi_0(\vec{x}), \pi_0(\vec{y}) \in S$ correspond to points $\vec{x}, \vec{y} \in C$ respectively, with the property that $\tau(\vec{y}) >  0$ and $\tau(\vec{x}) <  0$.
\end{cor}

\begin{defn}
Suppose $\vec{p}, \vec{q}$ are points on the loop, which are projected down to $p$ and $q$ respectively, to form piercings on the closed surface $\pi_0(S)$. Let $C$ be an arc in the loop, whose boundary points are $\vec{p}$ to $\vec{q}$. We say that these two piercings are consecutive if the projected arc $\pi_0(C)$ joining the piercings $p = \pi_0(\vec{p})$ and $q = \pi_0(\vec{q})$, lies strictly inside the interior or the exterior of $\pi_0(S)$.
\end{defn}

When a time-like loop $l$ is disjoint from a closed surface $S$, we want to show that the linking number is invariant under ambient isotopy. It suffices to deform $l$, while keeping $S$ unchanged. It is easy to see that any orientation preserving ambient isotopy that shifts the position of piercings inside $\pi_0(S)$ will not change the orientation and height of each piercing in $\pi_0(S)$. So, to show that the linking number is invariant under ambient isotopy, it suffices to show that it is invariant under Movement $W$.

\begin{thm}\label{t.w.2}
Orientate the time-like loop $l$ and assign a normal to the closed surface $S = \pi_0(S)$. The linking number between $l$ and $S$ computed using the projection $\pi_0$, is an invariant under Movement $W$.
\end{thm}

\begin{proof}
We want to show that the linking number remains invariant under Movement $W$. Suppose we have points $\vec{x}$ and $\vec{y}$ in $l$, such that $\pi_0(\vec{x})$ and $\pi_0(\vec{y})$ are consecutive piercings on $S$.

Without any loss of generality,
\begin{itemize}
  \item choose an outward pointing normal on $S$;
  \item let $C$ be an arc in the loop $l \subset \bR^4$, with end points $\vec{x}$ and $\vec{y}$, such that the arc $\pi_0(C)$ from $\pi_0(l)$, going from $\pi_0(\vec{x})$ and $\pi_0(\vec{y})$, lies in the interior of $S$.
\end{itemize}

Let $S' \subset S$ be homeomorphic to an open disc, such that $\pi_0(\vec{x}), \pi_0(\vec{y}) \in \pi_0(S')$. Since we are going to apply Movement $W$ to remove the piercings, we may as well assume that $\pi_0(S')$ lies above $\pi_0(C)$. The normal of $S'$, $n_{S'}$ will be pointing upwards and the tangent vector of $\pi_0(C)$ at $\pi_0(\vec{x})$, is in opposite direction of the tangent vector of $\pi_0(C)$ at $\pi_0(\vec{y})$. Therefore, the orientation at $\pi_0(\vec{x})$ is of the opposite sign of the orientation at $\pi_0(\vec{y})$.

By Lemma \ref{l.w.1}, we can remove these two piercings using Movement $W$, when either $\tau(\vec{x})$ and $\tau(\vec{y})$ are both less than zero or both more than zero. In either case, we note that the height at both $\pi_0(\vec{x})$ and $\pi_0(\vec{y})$ are the same.

Thus, $\pi_0(\vec{x})$ and $\pi_0(\vec{y})$ have different algebraic numbers, so both their algebraic piercing numbers do not contribute to the linking number. Hence, the linking number remains invariant by removing these two piercings, using Movement $W$.
\end{proof}

For any time-like hyperlink $L = \{l^1, \cdots, l^n\}$, we will now
define the linking number between $L$ and $S$, as \beq {\rm lk}(L, S) := \sum_{u=1}^n{\rm lk}(l^u, S). \nonumber \eeq It is invariant under Movement $W$ between $S$ and the component loops in $L$.

By applying Movement $W$ if necessary, Corollary \ref{c.w.1} says that we may and will assume that for any two consecutive piercings $p = \pi_0(\vec{p})$ and $q = \pi_0(\vec{q})$, we must have $\tau(\vec{p}) < 0  < \tau(\vec{q})$. Henceforth, we will call $p$ ($q$) a left (right) piercing and say that it forms before (after) the formation of the surface $S$. Thus if two piercings are consecutive, then we will assume that one must be a left and the other a right piercing. If an oriented arc $\pi_0(C)$ going from the left piercing $\pi_0(\vec{p})$ to the right piercing $\pi_0(\vec{q})$ lies in the interior (exterior) of $S$, then all the other oriented arcs going from a left piercing to a right piercing must also lie in the interior (exterior).

Recall we defined the linear transformation $T_c$ in Equation (\ref{e.tc.1}). As we only have a finite number of piercings, we can find a $\xi > 0$ such that for any $0\leq c \leq \xi$, $S_c := T_c(S)$ is a closed surface in $\bR^4$ with \beq \tau(\vec{p}) < \inf_{\vec{x} \in S_c}\tau(\vec{x}) \leq \sup_{\vec{x} \in S_c}\tau(\vec{x})  < \tau(\vec{q}) \label{e.m.1} \eeq for any left piercing $\pi_0(\vec{p}) \in \pi_0(S_c)$ and right piercing $\pi_0(\vec{q}) \in \pi_0(S_c)$. Denote $\underline{\delta} = \inf_{\vec{x} \in S_c}\tau(\vec{x})$ and $\overline{\delta} = \sup_{\vec{x} \in S_c}\tau(\vec{x})$.

Note that $\pi_0: (x_0, x) \in S_c \mapsto x \in \pi_0(S_c)$ is injective. By abuse of notation, we define $\tau: x \in \pi_0(S_c) \rightarrow  \tau(x) \in \bR$, for $(\tau(x),x) \in S_c$. For any time $t \in \bR$, we can interpret $\tau^{-1}(t) \subset \pi_0(S_c)$ as the set that is formed at time $t$. Thus, the map $t \mapsto \tau^{-1}(-\infty, t]$ will describe the formation of $\pi_0(S_c)$. Before time $\underline{\delta}$, we see that the surface $\pi_0(S_c)$ has yet to form. As time progresses from $\underline{\delta}$ to $\overline{\delta}$, we see that patches of $\pi_0(S_c)$ begin to take shape and the closed surface will be fully formed at time $\overline{\delta}$ and beyond. Therefore, we will say that $[\underline{\delta}, \overline{\delta}]$ is the time period for which the closed surface $\pi_0(S)$ is formed.

Equation (\ref{e.m.1}) will thus describe the following event which consists of 3 occurrences; a left piercing is first formed, followed by the formation of the closed surface, and finally a consecutive right piercing is formed. During the time period between the formation of the two consecutive piercings, the oriented arc joining the two piercings can either be in the interior or in the exterior.

\begin{notation}
We will say that all the points of entry into the interior of $\pi_0(S_c)$ happen before (after) the formation of the closed surface and denote it by $l< S_c$ ($l > S_c$), in the former (latter).
\end{notation}

\begin{rem}\label{r.rs.2}
\begin{enumerate}
  \item We see that there is an implicit time-ordering between the formation of the surface and the piercings, defined by $l<S_c$ or $l > S_c$.
  \item By changing the orientation of the loop, we can change this time-ordering.
\end{enumerate}
\end{rem}

\subsection{Surfaces with boundary}

We will now consider when an orientable surface $S \subset \bR \times \bR^3$ has a non-empty boundary, $\partial S$. Based on our earlier discussion, it suffices to consider that $S \subset \{0\} \times \bR^3$. So $S$ is a Seifert surface of its boundary $\partial S$ and an orientation on its boundary $\partial S$ is consistent with the orientation on $S$. If we have a loop $l$ in $\bR \times \bR^3$, under ambient isotopy, $l$ can be `unlink' from $\partial S$, so this means that the linking number between $S$ and $l$ will be trivial.

To obtain a non-trivial linking number, we may impose the condition that the boundary $\partial S$, together with the loop $l$, must be a time-like hyperlink, which we will denote as $\chi(l, \partial S)$. And any ambient isotopy of $l$ and $S$ should maintain $\chi(l, \partial S)$ as time-like. But there is no reason why by keeping the boundary and the loop time-like, the linking number between the surface and the loop will be well-defined.

\begin{defn}
Let $l$ be a time-like loop and $S = \bigcup_{k=1}^n S_k$, whereby $S_k$ is a connected compact surface with boundary $\partial S_k$. We say that $l$ is time-ordered with $S$, if $\tau(\vec{x}) < \tau(\vec{y})$ or $\tau(\vec{x}) > \tau(\vec{y})$, for every $\vec{x} \in l$ and $\vec{y} \in S_k$. We will denote this relation as $l < S_k$ ($l > S_k$) for the former (latter), by abuse of notation.
\end{defn}

When we project $S$ and $l$ using $\pi_0$ in $\bR^3$, note that $\pi_0(\chi(l,\partial S))$ will form a link in $\bR^3$, by definition of time-like hyperlink.

\begin{prop}
Let $l$ be a time-like loop in $\bR \times \bR^3$ and $S \subset \{0\} \times \bR^3$ is a surface with boundary $\partial S$. Orientate both $l$ and $S$, and hence the projected boundary $\partial S$ is assigned an orientation, consistent with the orientation of the surface $S$.

Suppose $l$ is time-ordered with $S$. The linking number between $l$ and $S$ is computed to be equal to ${\rm sk}(l, \partial S)/6$. If we compute the linking number between $l$ and $S$ using $\pi_i$, $i=1,2,3$, then it will be 0.
\end{prop}

\begin{proof}
It suffices to prove the case when $S$ is connected. Note that $S$ in $\bR^3$ is a Seifert surface, with boundary $\partial S$. From \cite{rolfsen1976knots}, we know that the linking number between $S$ and $\pi_0(l)$, computed by summing up only the orientation of piercings in $S$, is equal to 1/2 times the linking number between $\partial S$ and $\pi_0(l)$. Because of the time-ordering, all the piercings have the same height, hence ${\rm lk}(l, S) = {\rm sk}(l, \partial S)/6$, by the definition of the linking number given in this article.

Now consider using the projection $\pi_i$. But because $\tau(\vec{x}) < \tau(\vec{y})$ or $\tau(\vec{x}) > \tau(\vec{y})$ for every $\vec{x} \in l$ and $\vec{y} \in S$, hence $\pi_i(l)$ will not intersect $\pi_i(S)$. So there is no piercing, therefore the linking number is computed to be 0 in this case.
\end{proof}

\begin{defn}\label{d.ts.2}
Start with a time-like, oriented loop $l$, disjoint from an orientable compact surface $S \subset \bR^3$ with boundary $\partial S$, also assigned with an orientation. We also assume that $\partial S$, together with $l$, is a time-like oriented hyperlink, denoted as $\chi(l, \partial S)$. We say that $(l,S)$ is time-like isotopic to an oriented pair $(l',S')$, preserving the time-ordering, if there is an orientation preserving continuous map $F$ as described in Definition \ref{d.ts.1}, such that
\begin{itemize}
  \item $F_1(S) = S'$, $F_1(\chi(l, \partial S)) = \chi(l', \partial S')$, and during the ambient isotopy process, $F_t(\chi(l, \partial S))$ remains a time-like hyperlink;
  \item the time-ordering between intermediate deformations $F_t(l)$ and $F_t(S)$ remains unchanged. Hence $l < S$ if and only if $l' < S'$, provided $S$ is connected.
\end{itemize}
We can then define an equivalence relation and hence define an equivalence class containing a pair $(l, S)$.
\end{defn}

\begin{rem}
Here, we impose time-ordering, which implies causality, and the equivalence relation ensures that it is not violated under the isotopy process.
\end{rem}

For an oriented equivalence class $[(l,S)]$ (both assigned with an orientation), we will now define the linking number between a time-like loop $l$ and an orientable compact surface $S$ with boundary, as \beq {\rm lk}(l, S) :=  {\rm sk}(l, \partial S)/6.
\label{e.lk.1} \eeq By definition of the equivalence relation, this is well-defined.

For an oriented time-like hyperlink $L = \{l^1, \cdots, l^n\}$, we say an oriented pair $(L, S)$ is time-like isotopic, preserving the time-ordering, to an oriented pair $(L',S')$, $L' = \{l^{\prime, 1}, \cdots, l^{\prime,n} \}$, if there is an orientation preserving continuous map $F$ as described in Definition \ref{d.ts.1}, with $F_1(L) = L'$, $F_1(S) = S'$, and showing that
\begin{itemize}
  \item $L$ is time-like isotopic to $L'$;
  \item for each component $l^u \in L$, we have $(l^u, S)$ is time-like isotopic, preserving the time-ordering, to $(l^{\prime,u}, S')$, for each $u=1,2, \cdots, n$.
\end{itemize}
By definition, we require for each $u = 1, \cdots , n$, $l^u$ is time-ordered with $S$.

Define the linking number between an oriented time-like hyperlink $L$ and an oriented surface $S$ with boundary, as \beq {\rm lk}(L, S) := \sum_{u=1}^n{\rm lk}(l^u, S), \nonumber \eeq ${\rm lk}(l^u, S)$ as defined in Equation (\ref{e.lk.1}). Then this linking number as defined, is invariant under any orientation preserving, ambient isotopy of $L$ and $S$, as long as the process keeps the deformed hyperlink and the boundary together as an oriented time-like hyperlink, while preserving the time-ordering between the component loops and the connected components in the surface.

\subsection{Piercing number}

Let $L$ be a time-like hyperlink and $S\subset \{0\} \times \bR^3$ be a compact surface, with or without boundary, and both are disjoint. Let $\pd(\pi_0; L, S)$ denote the set of piercings between $\pi_0(L)$ and $S\equiv \pi_0(S)$, and $\left|\pd(\pi_0; L, S)\right|$ is the total number of piercings in the set. Now, $\left|\pd(\pi_0; L, S)\right|$ is not invariant under ambient isotopy of a hyperlink $L$ and a surface $S$, since using Movement $W$, we can introduce more piercings.

\begin{defn}(Piercing number)\label{d.pn.1}\\
Suppose a surface $S \subset \{0\} \times \bR^3$ has no boundary and $L$ is a time-like hyperlink, disjoint from $S$. We say $(L', S') \sim (L, S)$, if $L \cup S$ are ambient isotopic to $L' \cup S'$, with $L$ time-like isotopic to $L'$.

If $S$ has boundary $\partial S$, $(L', S') \sim (L, S)$ will mean the time-like isotopy, preserving time-ordering, equivalence relation as described earlier.

The piercing number between $L$ and $S$, denoted as $\nu_S(L)$, will be defined to be the infimum of $\left|\pd(\pi_0; L', S')\right|$, taken over all possible pairs $(L',S')$ which are equivalent to $(L,S)$ under ambient isotopy when $S$ has no boundary; under the time-like isotopy, preserving time-ordering, when $S$ has boundary.
\end{defn}

When $S$ has no boundary, the piercing number will be an invariant between $L$ and $S$ under ambient isotopy. No time-ordering is required to be imposed.

As discussed earlier, the linking number between an oriented hyperlink with an oriented surface with boundary is not a topological invariant, but an invariant under time-like isotopy, preserving time-ordering, as given in Definition \ref{d.ts.2}. Hence, the piercing number will also be an invariant under this equivalence relation. Because we are using $\pi_0$ as the projection, we will obtain a non-trivial piercing number between a hyperlink and a surface.

\section{Framed hyperlinks}

Let $v$ be a non-tangential vector field on an oriented knot $\gamma$. Now shift the knot along $\epsilon v$, whereby $\epsilon > 0$ and is small. Call this shifted curve $\gamma^\epsilon := \gamma + \epsilon v$. A framing of a knot is a homotopy class of normal vector fields on $\gamma$ where two normal vector fields are said to be homotopic if they can be deformed into one another within the class of normal vector fields. Thus, the vector field $v$ defines a frame for the knot $\gamma$ and $\{\gamma, v\}$ is called a framed knot or ribbon.

Now consider $\gamma$ and $\gamma^\epsilon$ as separate knots. Project it down onto $\Sigma_i$ plane to form a link diagram as above. A half-twist is formed when a displaced copy of an arc inside $\gamma^\epsilon$ twirls around the original arc in $\gamma$, which projects onto a plane to form a crossing $q$. This is analogous to the two arcs forming the outline of a thin strip piece of paper and twisting it by $\pi$, giving us a half twist. Thus, we can define the algebraic crossing of the half-twist as $\varepsilon(q)$. For a more detailed description of a half-twist, we refer the reader to \cite{CS-Lim02}.

A framed link $L \subset \bR^3$ will be a finite set of non-intersecting simple closed curves, whereby each component knot is equipped with a frame. We can project it onto a plane as described above. Two link diagrams represent the same framed link, if one diagram can be obtained from the other diagram, by a sequence of Reidemeister moves ${\rm I}^\prime$, II and III. Reidemeister move ${\rm I}^\prime$ says that twisting in one direction, followed by twisting in the opposite direction, undo both twists. See page 271 in \cite{MR1321145}, Figure 8.2, which also shows a full twist. A full twist is just two consecutive half twists, twisted in the same direction, placed together. In any planar diagram of a framed link, we can always assume that half twists occur in pairs, thus forming a full twist. Two framed links $L$ and $L'$ are ambient isotopic to each other if when projected on the same plane to form link diagrams $D$ and $D'$ respectively, we can obtain $D'$ from $D$ by a finite sequence of Reidemeister moves.

Since twisting a ribbon in the positive direction, followed by twisting it in the negative direction, undo all the twists, hence we can and will assume that all the half-twists on a knot have the same algebraic number. Therefore, we can and will always assume that the total number of half-twists is a minimum on a framed knot. Obviously, the set of half-twists in a link-diagram will depend on which plane we project the framed knot on. However, the total number of half-twists will be the same and is an even number, independent of the plane we choose.

A framed hyperlink will be a hyperlink whereby each component loop $l^u$, when projected in $\bR^3$ to form a knot $\pi_0(l^u)$, is equipped with a frame. When we say a framed time-like hyperlink $L$ is time-like isotopic to a framed time-like hyperlink $L'$, we mean that $L$ and $L'$ are time-like isotopic as in Definition \ref{d.ts.1} and furthermore, $\pi_0(L)$ and $\pi_0(L')$ are ambient isotopic as framed links in $\bR^3$.

Now half-twists lie in a link diagram. However, we can `lift' these half-twists and represent them as nodes on $\pi_0(l)$. Therefore, in future, we will now view a framed knot as a knot in $\bR^3$, but with nodes attached to it. We can define the algebraic number of a node to be equal to the algebraic crossing number of its corresponding half-twist.

For a framed loop $l$, let ${\rm Nd}(\pi_0(l))$ be the set of nodes on a projected loop $\pi_0(l)$, assuming all of them have the same sign for its algebraic number. This means that we have the minimum number of nodes on the framed knot. Equivalence class of framed loop $l$ allows us to move the nodes along the framed knot $\pi_0(l)$, so we should view ${\rm Nd}(\pi_0(l))$ as an equivalence class.

Let $R$ be a bounded and possibly disconnected 3-dimensional submanifold inside spatial space $\bR^3$, containing all of its boundary. We will term $R$ as a compact solid region and view $R \subset \{0\} \times \bR^3$. Its boundary $\partial R$ will be a closed surface.

There is no notion of a linking number between a loop $l$ and a bounded 3-dimensional region $R \subset \bR^3$ inside $\bR^4$. But recall we always assume that all the nodes on $\pi_0(l)$ have the same sign, so indeed we will have the minimum even number of nodes. Therefore, we can define the confinement number between $l$ and $R$, by counting how many of the nodes in ${\rm Nd}(\pi_0(l)) \subset \bR^3$, viewed as being 0-dimensional manifold, lie in the interior of a compact region $R$.

Suppose a loop $l$ is disjoint from a compact region $R \subset \{0\} \times \bR^3$. If there is an arc $C$ in $l$ such that $\pi_0(C) \subset R$, then we must have that $\tau(\vec{x}) > 0$ or $\tau(\vec{x}) < 0$ for every $\vec{x} \in C$. By applying Movement $W$, we can slowly withdraw $\pi_0(C)$ out of the interior of $\partial R$, so under ambient isotopy, the projection $\pi_0$ of a loop $l'$ equivalent to $l$ will be disjoint from $\pi_0(R)$. The nodes of a framed knot $\pi_0(l')$ will not lie in the interior, which will give us a trivial confinement number. Hence we need to introduce the following definition of an equivalence class.

\begin{defn}\label{d.ts.3}
Suppose we have a pair $(L, R)$, whereby $L$ is a time-like hyperlink, disjoint from a compact solid region $R \subset \{0\} \times\bR^3$. Assume that each component loop $l^u \in L$ is a framed loop, i.e. we have nodes on $\pi_0(l^u)$, possibly empty, and all the nodes have the same algebraic sign. We say that $(L, R)$ is time-like isotopic to $(L', R')$ if there is a continuous map $F$ as defined in Definition \ref{d.ts.1}, such that
\begin{itemize}
  \item $F_1(L) = L'$ and for every $t$, $F_t(L)$ is a framed hyperlink;
  \item $F_1(R) = R' \in \{0\} \times \bR^3$ and for every $t$, $F_t(R) \subset \{0\} \times \bR^3$;
  \item $F_t(x) \notin \partial R$, for each node $x\in \pi_0(l^u)$.
\end{itemize}
Using the above isotopy, we can define an equivalence relation and hence an equivalence class containing $(L, R)$.
\end{defn}

\begin{rem}
The last property says that during the time-like isotopy, the nodes on the deformed loop $F_t(\pi_0(l^u))$ are not allowed to cross the boundary $\partial R$, which is a closed surface.
\end{rem}

\begin{defn}(Confinement number)\\
Let $R$ be a compact solid region in $\{0\} \times \bR^3$, $L= \{l^1, \ldots, l^n\}$ be a framed hyperlink, disjoint from $R$. Define ${\rm Nd}(\pi_0(l^u)) \subset \bR^3$, the finite set containing the minimum number of nodes on the projected loop $\pi_0(l^u) \subset \bR^3$, $l^u \in L$.
Thus, all the nodes have the same algebraic sign.

For an equivalence class $[(L, R)]$ as defined in Definition \ref{d.ts.3}, we define the confinement number, $\nu_R(L)$, as \beq \nu_R(L) := \sum_{u=1}^n \nu_R(l^u), \nonumber \eeq whereby $\nu_R(l^u)$ is equal to the number of nodes in ${\rm Nd}(\pi_0(l^u))$ that lie inside the interior of $R$.
\end{defn}

\begin{rem}
\begin{enumerate}
  \item If there is no frame assigned, hence no nodes, then the set of nodes will be the empty set. By definition, we only consider the minimum number of nodes on the knot, so we are not allowed to add more half-twists.
  \item Notice that we count the number of nodes from the framed knot, which lie inside the interior of $R$, under the assumption that the set ${\rm Nd}(\pi_0(l))$ is disjoint from the boundary of $\partial R$ of $R$. We allow the knot $\pi_0(l)$ and the region $R$ to be deformed up to ambient isotopy, as long as the set ${\rm Nd}(\pi_0(l))$ remains disjoint from the boundary $\partial R$.
\end{enumerate}
\end{rem}

\begin{prop}
The confinement number is an invariant, under the equivalence relation given in Definition \ref{d.ts.3}.
\end{prop}

\begin{proof}
Let $[(l, R)] = [(l', R')]$ be an equivalence class, $l$ is a loop. It suffices to show that $\nu_R(l) = \nu_{R}(l')$ for a framed loop. Let $F$ be a continuous map as described in Definition \ref{d.ts.3}, such that $F_1(l) = l'$ and $F_1(R) = R'$. Because $F_t$ is a homeomorphism, it maps ${\rm Nd}(\pi_0(l))$ to ${\rm Nd}(\pi_0(l'))$ bijectively. The nodes that lie in the interior of $R$ can only be mapped to nodes in the interior of $R'$ and vice versa, again by the property of $F$. Thus, $\nu_R(l) = \nu_{R}(l')$.
\end{proof}

When there is an arc $C$ such that $\pi_0(C)$ is in the interior of $R$, then we note that this arc must occur before or after time $x_0 = 0$, implying causality. Thus, the nodes on $\pi_0(C)$, if any, must be time-ordered with the region $R$. By definition of the equivalence relation, the time occurrence of these nodes cannot change. Hence, causality is preserved under time-like isotopy.

\section{Summary}

We would like to conclude by summarizing all the facts we have have discussed. Our ambient space is $\bR \times \bR^3$, a 4-manifold. We consider the following submanifolds.

Let $\oL$ and $\uL$ be two distinct oriented hyperlinks, the former will be termed matter hyperlink, and latter termed as geometric hyperlink. The hyperlinks are expected to be time-like and for $\ol^u \in \oL$ and $\ul^v\in \uL$, we have either $\ol^u < \ul^v$ or $\ol^u > \ul^v$. Together, they form an oriented time-like hyperlink $\chi(\oL, \uL)$, consisting of time-ordered pair of hyperlinks.

Equip the matter hyperlink with a frame, so that $\pi_0(\oL)$ is a framed link. That is, add in nodes to $\pi_0(\oL)$. Further assume that we have the minimum number of nodes, which is equivalent to all the nodes from the same component knot in $\pi_0(\oL)$ having the same algebraic sign. The set of nodes should be thought of as an equivalence class, denoted by ${\rm Nd}(\pi_0(\oL))$.

Introduce an (possibly disconnected) oriented compact surface $S$ in $\{0\} \times \bR^3 \subset \bR^4$, with or without boundary. If it has non-empty boundary $\partial S$, then $\partial S \cup \chi(\oL, \uL)$ together must form a time-like oriented hyperlink and furthermore, each component loop in $\chi(\oL, \uL)$ is time-ordered with each connected component in $S$ with boundary. Finally we have a (possibly disconnected) compact solid region $R$ in $\{0\} \times \bR^3 \subset \bR^4$. We assume that $\chi(\oL, \uL)$ do not intersect $S$ and $R$.

Let us summarize all the invariants we have discussed in the following theorem.

\begin{thm}
Consider the oriented triple $\{S, R, \chi(\oL, \uL)\}$ as described above. Note that $\chi(\oL, \uL)$ do not intersect $S$ and $R$. We define a time-like and time-ordered equivalence relation, and say $\{S, R, \chi(\oL, \uL)\}$ is time-like isotopic to $\{\mathcal{S}, \mathcal{R}, \chi(\overline{\mathcal{L}}, \underline{\mathcal{L}})\}$, preserving the time-ordering, if there is an orientation preserving continuous map $F$ as described in Definition \ref{d.ts.1}, with $F_1(\chi(\oL, \uL)) = \chi(\overline{\mathcal{L}}, \underline{\mathcal{L}})$, $F_1(S) = \mathcal{S}$ and $F_1(R) = \mathcal{R}$, and showing that
\begin{enumerate}
  \item $\chi(\oL, \uL) $ is time-like isotopic, preserving the time-ordering, to $\chi(\overline{\mathcal{L}}, \underline{\mathcal{L}})$;
  \item $(\chi(\oL, \uL), S) \sim (\chi(\overline{\mathcal{L}}, \underline{\mathcal{L}}), \mathcal{S})$ as defined in Definition \ref{d.pn.1};
  \item $(\oL,R)$ is time-like isotopic to $(\overline{\mathcal{L}},\mathcal{R})$.
\end{enumerate}

If the above conditions are met, then we will call this equivalence class a time-like triple. Furthermore,
\begin{enumerate}
  \item the hyperlinking number ${\rm sk}(\oL, \uL)$ of $\chi(\oL, \uL)$;
  \item the linking number ${\rm lk}(\uL, S)$ between $\uL$ and $S$;
  \item the piercing number $\nu_S(\oL)$ between $\oL$ and $S$;
  \item the confinement number $\nu_R(\oL)$ between the framed hyperlink $\oL$ and $R$,
\end{enumerate}
are all invariant under time-like isotopy, preserving the time-ordering, as described above.

%
\end{thm}

\begin{rem}
Note that linking numbers for component knots in $\pi_0(\oL)$ do not appear in quantum geometry. Instead, they showed up in quantized Chern-Simons theory, as proved in  \cite{CS-Lim01}. The idea of using Chern-Simons theory to obtain knot invariants was first described in \cite{MR990772}.
\end{rem}

In the definition of the equivalence relation, we insist that time-ordering is preserved, between loops. On careful thought, this restriction is indeed necessary. Causality is respected in special relativity, i.e. cause and effect cannot be interchanged. Smolin in \cite{Smolin:2005mq}, stated that General Relativity is a physical relational theory, whereby causal order of events has to be respected.

Consider the equivalence class containing a matter loop and a geometric loop, for simplicity. In the equivalence class, we consider either the matter loop occurs before or after a geometric loop. This means two things. First, a cause and effect is implied in this class. Secondly, this cause and effect must be preserved under any homeomorphism of space-time, as required by the equivalence relation. We are still not allowed to switch their causality, even though an ambient isotopy between them may exist and it ensures they are time-like throughout the process. Hence causality between matter and geometric hyperlinks becomes a necessary condition under the equivalence relation.

Throughout the article, we introduced geometric objects like a compact surface and solid region. We also remarked that time-ordering is either implicitly or explicitly defined, hinting that a cause and effect event is at play. Furthermore, causality is preserved under the appropriate equivalence relation.

Quantum geometry was developed in the mid-nineties by several researchers and a detailed reference can be found in \cite{0264-9381-21-15-R01}, \cite{rovelli2004quantum} and \cite{Thiemann:2007zz}. From these articles, one can see that quantum geometry is in a way, describing the discretization of space-time. In \cite{Rovelli_2011}, Rovelli described quantum geometry as a subject on quanta of space-time. Using canonical quantization, he and his fellow co-author quantized area of a surface, which counts the number of times a spin network graph intersects a surface. See \cite{rovelli1995discreteness}. Their result was the first to show how the piercing number, can be used to compute area, in quantum geometry.

But make no mistake. It is not discrete geometry. In \cite{Ashtekar:2000eq}, the authors describe quantum geometry as a theory of interaction between quantum matter and geometry. In this article, quantum matter is represented by a matter hyperlink, and geometry is represented by a geometric hyperlink, a surface and a solid region. More generally, quantum geometry is a form of topological theory, which focuses on how submanifolds are `linked' in $\bR^4$ or any globally hyperbolic 4-manifold, rather than the ambient manifold itself. It was also remarked in \cite{Smolin:2005mq} that one should not focus on the topology on the ambient space, but rather how the events are causally linked.

But what sets it apart from other topological theories? In topology, area, volume and curvature have no meaning. One needs to define a metric or a connection to define these quantities. Unlike topological theory, in quantum geometry, we do have notions of area, volume and curvature in quantum geometry, without using a metric or connection.

In Loop Quantum Gravity, one uses the Einstein-Hilbert action to define a path integral. See \cite{EH-Lim02}. By averaging area of the surface or volume of a region over all (degenerate) metric, one can quantize the area and volume into its corresponding operators. This was done in \cite{EH-Lim03} and \cite{EH-Lim04} respectively. The eigenvalues are computed from the piercing and confinement numbers respectively, which are invariants of the time-like triple discussed above. The discrete eigenvalues hence show that space-time is discretized.

In a similar manner, we can quantize curvature of a surface by averaging over all connections on ambient space, into an operator. Quantized curvature now becomes an invariant under an equivalence relation, computed using the linking number between a surface and a hyperlink. See \cite{EH-Lim05}.

In quantum geometry, there is no preferred choice of metric or connection, hence no classical background geometry is introduced. See \cite{Smolin:2005mq}. So, area and curvature of a surface $S$ and the volume of a compact region $R$, represented up to ambient isotopy, are not defined in the classical sense. The time-like triple $\{S, R, \chi(\oL, \uL)\}$ considered in quantum geometry gives meaning to area, volume and curvature and turn these physical quantities into invariants, under an equivalence relation.



\end{document}